\documentclass[10pt,fleqn]{amsart}

\usepackage{amsmath,amssymb,latexsym}
\usepackage[mathscr]{eucal}

\usepackage{cite}

\theoremstyle{plain}
\newtheorem{theorem}{Theorem}

\numberwithin{equation}{section}

\newcommand{\dy}{\partial}
\newcommand{\ddt}[1]{\frac{\mathrm{d}{#1}}{\mathrm{d}{t}}}

\newcommand{\sfrac}[2]{{\textstyle\frac{#1}{#2}}}
\newcommand{\tssum}{{\textstyle\sum}}
\newcommand{\nlsum}{{\sum\nolimits}}
\newcommand{\tssup}{{\textstyle\sup}}

\newcommand{\ex}{\mathrm{e}}

\newcommand{\eps}{\varepsilon}

\newcommand{\ncdot}{\!\cdot\!}

\newcommand{\gb}{\nabla}

\newcommand{\nk}[1]{{\lfloor #1\rfloor}}
\newcommand{\Nk}[1]{{\lceil #1\rceil}}

\newcommand{\aand}{\quad\textrm{and}\quad}

\newcommand{\circled}[1]{\textrm{\textcircled{\raisebox{-0.1ex}{\footnotesize{#1}}}}}

\newcommand{\Dom}{\mathscr{D}}

\newcommand{\ff}{Q}
\newcommand{\parm}{\pi}

\newcommand{\gn}[1]{|\![{#1}]\!|}

\newcommand{\cpoi}{c_0^{}}

\newcommand{\vb}{\boldsymbol{v}}
\newcommand{\xb}{\boldsymbol{x}}

\newcommand{\psih}{\hat\psi}
\newcommand{\wh}{\hat\omega}
\newcommand{\zh}{\boldsymbol{e}_z}

\newcommand{\ecT}{\alpha_T}
\newcommand{\ecS}{\alpha_S}

\newcommand{\Sh}{\hat S}
\newcommand{\Th}{\hat T}
\newcommand{\Uh}{\hat U}

\newcommand{\dt}{\;\mathrm{d}t}
\newcommand{\dx}{\;\mathrm{d}x}

\newcommand{\kv}{\kappa_v}
\newcommand{\kt}{\kappa_T}
\newcommand{\ks}{\kappa_S}

\newcommand{\flt}{Q_T}
\newcommand{\fls}{Q_S}
\newcommand{\flu}{Q_u}
\newcommand{\fla}{Q}

\newcommand{\tq}{T_Q}
\newcommand{\sq}{S_Q}

\newcommand{\ppr}{\mathfrak{p}}
\newcommand{\pts}{\beta}

\newcommand{\xnd}{\xi}

\newcommand{\Attr}{\mathcal{A}}
\newcommand{\dds}{\mathbb{S}}

\newcommand{\mw}{M_\omega}
\newcommand{\mt}{\tilde M}

\newcommand{\LIM}{\mathop{\hbox{\sc lim}}}
\newcommand{\dmu}{\;\mathrm{d}\mu}

\begin{document}

\title[2d double-diffusive convection]%
{Long-time dynamics of 2d double-diffusive convection: analysis and/of numerics}

\author[Tone]{Florentina~Tone}
\email{ftone@uwf.edu}
\address[FT]{Department of Mathematics and Statistics\\
  University of West Florida\\
  Pensacola, FL~32514, United States}

\author[Wang]{Xiaoming~Wang}
\email{wxm@mail.math.fsu.edu}
\urladdr{http://www.math.fsu.edu/\~{}wxm}
\address[XW]{Department of Mathematics\\
   Florida State University\\
   Tallahassee, FL 32306--4510, United States}

\author[Wirosoetisno]{Djoko~Wirosoetisno}
\email{djoko.wirosoetisno@durham.ac.uk}
\urladdr{http://www.maths.dur.ac.uk/\~{}dma0dw}
\address[DW]{Mathematical Sciences\\
   Durham University,
   Durham\ \ DH1~3LE, United Kingdom}

\thanks{Wang's work is supported in part by grants from NSF and a planning
grant from FSU}


\begin{abstract}
We consider a two-dimensional model of double-diffusive convection and
its time discretisation using a second-order scheme which treat the
nonlinear term explicitly (backward differentiation formula with
a one-leg method).
Uniform bounds on the solutions of both the continuous and discrete
models are derived (under a timestep restriction for the discrete
model), proving the existence of attractors and invariant measures
supported on them.
As a consequence, the convergence of the
attractors and long time statistical properties
of the discrete model to those of the continuous one in
the limit of vanishing timestep can be obtained following established methods.
\end{abstract}

\maketitle


\section{Introduction}\label{s:intro}

The phenomenon of double-diffusive convection, in which two properties
of a fluid are transported by the same velocity field but diffused at
different rates, often occur in nature \cite{huppert-turner:81}.
Perhaps the best known example is the transport throughout the world's
oceans of heat and salinity, which has been recognised as an essential
part of climate dynamics \cite{schmitt:94,weaver-bitz-fanning-holland:99}.
In contrast to simple convections (cf.\ \cite{chandrasekhar:hdhms}),
double-diffusive convections support a richer set of physical regimes,
e.g., a stably stratified initial state rendered unstable by diffusive effects.
Although in this paper we shall be referring to the
oceanographic case, the mathematical theory is essentially identical
for astrophysical \cite{spiegel:69,mirouh-al:12}
and industrial \cite{chen-johnson:84} applications.


In this paper, we consider a two-dimensional double-diffusive convection
model, which by now-standard techniques \cite{temam:iddsmp} can be proved
to have a global attractor and invariant measures supported on it,
and its temporal discretisation.
We use a backward differentiation formula for the time derivative
and a fully explicit one-leg method \cite{hairer-wanner:sode2}
for the nonlinearities,
resulting in an accurate and efficient numerical scheme.
Of central interest, here and in many practical applications, is the
ability of the discretised model to capture long-time behaviours of the
underlying PDE.
This motivates the main aim of this article: to obtain bounds necessary
for the convergence of the attractor and associated invariant measures
of the discretised system to those of the continuous system.
We do this using the framework laid down in \cite{wxm:10,wxm:12}, with
necessary modifications for our more complex model.

For motivational concreteness, one could think of our system as a model
for the zonally-averaged thermohaline circulation in the world's oceans.
Here the physical axes correspond to latitude and altitude,
and the fluid is sea water whose internal motion is largely driven by
density differentials generated by the temperature $T$ and salinity $S$,
as well as by direct wind forcing on the surface.
Both $T$ and $S$ are also driven from the boundary---%
by precipitation/evaporation and ice melting/formation for the salinity,
and by the associated latent heat release and direct heating/cooling for
the temperature.
Physically, one expects the boundary forcing for $T$, $S$ and the momentum
to have zonal (latitude-dependent) structure, so we include these in our model.
Furthermore, one may also wish to impose a quasi-periodic time dependence
on the forcing;
although this is eminently possible, we do not do so in this paper to avoid
technicalities arising from time-dependent attractors.

\medskip
Taking as our domain $\Dom_*=[0,L_*]\times[0,H_*]$ which is periodic
in the horizontal direction, we consider a temperature field $T_*$
and a salinity field $S_*$, both transported by a velocity field
$\vb_*=(u_*,w_*)$ which is incompressible, $\gb_*\ncdot\vb_*=0$,
and diffused at rates $\kt$ and $\ks$, respectively,
\begin{equation}\begin{aligned}
   &\dy T_*/\dy t_* + \vb_*\ncdot\gb_* T_* = \kt\Delta_* T_*\\
   &\dy S_*/\dy t_* + \vb_*\ncdot\gb_* S_* = \ks\Delta_* S_*.
\end{aligned}\end{equation}
Here the star${}_*$ denotes dimensional variables.
Taking the Boussinesq approximation and assuming that the density
is a linear function of $T_*$ and $S_*$, which is a good approximation for
sea water (although not for fresh water near the freezing point),
the velocity field evolves according to
\begin{equation}
   \dy\vb_*/\dy t_* + \vb_*\ncdot\gb_*\vb_* + \gb_* p_* = \kv\Delta_*\vb_* + (\ecT T_* - \ecS S_*)\zh
\end{equation}
for some positive constants $\ecT$ and $\ecS$.

Our system is driven from the boundary by the heat and salinity fluxes
(which could be seen to arise from direct contact with air and latent
heat release in the case of heat, and from precipitation, evaporation
and ice formation/melt in the case of salinity),
\begin{equation}
   \dy T_*/\dy n_* = \flt{}_*
   \aand
   \dy S_*/\dy n_* = \fls{}_*.
\end{equation}
Here $n_*$ denotes the outward normal, $n_*=z_*$ at the top boundary
and $n_*=-z_*$ at the bottom boundary.
We also prescribe a wind-stress forcing,
\begin{equation}
   \dy u_*/\dy n_* = \flu{}_*
\end{equation}
along with the usual no-flux condition $w_*=0$ on $z_*=0$ and $z_*=H_*$.

\medskip
Largely following standard practice, we cast our system in non-dimensional
form as follows.
Using the scales $\tilde t$, $\tilde l$, $\tilde T$ and $\tilde S$,
we define the non-dimensional variables $t=t_*/\tilde t$, $\xb=\xb_*/\tilde l$,
$\vb=\vb_*\tilde t/\tilde l$, $T=T_*/\tilde T$ and $S=S_*/\tilde S$,
in terms of which our system reads
\begin{equation}\label{q:dUdt}\begin{aligned}
   \ppr^{-1}\bigl(&\dy_t\vb + \vb\ncdot\gb\vb\bigr)
	= - \gb p + \Delta\vb + (T-S)\zh\\
   &\dy_t T + \vb\ncdot\gb T = \Delta T\\
   &\dy_t S + \vb\ncdot\gb S = \pts\Delta S.
\end{aligned}\end{equation}
To arrive at this, we have put $\tilde l=H_*$ and taken the thermal
diffusive timescale for
\begin{equation}
   \tilde t = \tilde l^2/\kt,
\end{equation}
as well as scaled the dependent variables as
\begin{equation}
   \tilde T = \ppr\tilde l/(\ecT\tilde t^2)
   \aand
   \quad \tilde S = \ppr\tilde l/(\ecS\tilde t^2),
\end{equation}
where the non-dimensional {\em Prandtl number\/} and {\em diffusivity ratio\/}
(also known as the {\em Lewis number\/} in the engineering literature) are
\begin{equation}
   \ppr = \kv/\kt
   \aand
   \pts = \kt/\ks.
\end{equation}
Another non-dimensional quantity is the domain aspect ratio $\xnd=L_*/\tilde l$.
The surface fluxes are non-dimensionalised in the natural way:
$\flt=\ppr\flt{}_*/(\ecT\tilde t^2)$, $\fls=\ppr\fls{}_*/(\ecS\tilde t^2)$
and $\flu=\flu{}_*\tilde t$.

For clarity and convenience, keeping in mind the oceanographic application,
we assume that the fluxes vanish on the bottom boundary $z=0$,
\begin{equation}\label{q:BC0}
   \flu(x,0) = \flt(x,0) = \fls(x,0) = 0.
\end{equation}
For boundedness of the solution in time, the net fluxes must vanish, so
\eqref{q:BC0} then implies that the net fluxes vanish on the top boundary
$z=1$,
\begin{equation}\label{q:noflux}
   \int_0^\xnd \flu(x,1) \dx =
   \int_0^\xnd \flt(x,1) \dx =
   \int_0^\xnd \fls(x,1) \dx = 0.
\end{equation}
These boundary conditions can be seen to imply that the horizontal velocity
flux is constant in time, which we take to be zero, viz.,
\begin{equation}\label{q:uflux}
   \int_0^1 u(x,z,t) \;\mathrm{d}z
   = \int_0^1 u(x,z,0) \;\mathrm{d}z \equiv 0
   \quad\textrm{for all }x\in[0,\xnd].
\end{equation}
For some applications (e.g., the classical Rayleigh--B{\'e}nard problem),
the fluxes on the bottom boundary may not vanish, which must then be
balanced by the fluxes on the top boundary,
\begin{equation}
   \int_0^\xnd [\flt(x,1)-\flt(x,0)] \dx = 0
\end{equation}
and similarly for $\flu$ and $\fls$.
With some modifications (by subtracting background profiles from
$u$, $T$ and $S$), the analysis of this paper also apply to this
more general case.
This involves minimal conceptual difficulty but adds to the clutter,
so we do not treat this explicitly here.

Defining the vorticity $\omega:=\dy_xw-\dy_zu$, the streamfunction
$\psi$ by $\Delta\psi=\omega$ with $\psi=0$ on $\dy\Dom$
(this is consistent with \eqref{q:uflux}),
and the Jacobian $\dy(f,g):=\dy_xf\dy_zg-\dy_xg\dy_zf=-\dy(g,f)$,
our system reads
\begin{equation}\label{q:dudt}\begin{aligned}
   &\ppr^{-1}\bigl\{\dy_t\omega + \dy(\psi,\omega)\bigr\} = \Delta\omega + \dy_xT - \dy_xS\\
   &\dy_t T + \dy(\psi,T) = \Delta T\\
   &\dy_t S + \dy(\psi,S) = \pts\Delta S.
\end{aligned}\end{equation}
The boundary conditions are,
\begin{equation}\label{q:BC}
   \dy_z T = \flt,\quad
   \dy_z S = \fls,\quad
   \omega = \flu
   \aand
   \psi = 0
   \quad\textrm{on }\dy\Dom.
\end{equation}
In the rest of this paper, we will be working with \eqref{q:dudt}--\eqref{q:BC}
and its discretisation.
We assume that $\omega$, $T$ and $S$ all have zero integral over $\Dom$
at $t=0$.
Thanks to the no-net-flux condition \eqref{q:noflux},
this persists for all $t\ge0$.

Another dimensionless parameter often considered in studies of
(single-species) convection is the {\em Rayleigh number\/} Ra.
When the top and bottom temperatures are held at fixed values $T_1$
and $T_0$, Ra is proportional to $T_0-T_1$.
The relevant parameters in our problem would be
Ra${}_T\propto |\flt|_{L^2(\dy\Dom)}^{}$ and Ra${}_S\propto |\fls|_{L^2(\dy\Dom)}^{}$,
but we will not consider them explicitly here;
see, e.g., (2.11) in \cite{balmforth-al:06}.
For notational conciseness, we denote the variables $U:=(\omega,T,S)$,
the boundary forcing $Q:=(\flu,\flt,\fls)$ and the parameters
$\parm:=(\ppr,\pts,\xnd)$.


We do not provide details on the convergence of the global attractors and
long time statistical properties. Such kind of convergence can be obtained
by following established methods once we have the uniform estimates derived
here.
See \cite{hill-suli:95a} for the convergence of the global attractors and
\cite{wxm:12} for the convergence of long time statistical properties.

\medskip
The rest of this paper is structured as follows.
In section~\ref{s:cts} we review briefly the properties of the continuous
system, setting up the scene and the notation for its discretisation.
Next, we describe the time discrete system and derive uniform bounds
for the solution.
In the appendix, we present an alternate derivation of the boundedness
results in \cite{wxm:12}, without using Wente-type estimates but requiring
slightly more regular initial data.


\section{Properties of the continuous system}\label{s:cts}

In this section, we obtain uniform bounds on the solution of our system
and use them to prove the existence of a global attractor $\Attr$.
For the single diffusion case (of $T$ only, without $S$), this problem
has been treated in \cite{foias-manley-temam:87b} which we follow in
spirit, though not in detail in order to be closer to our treatment
of the discrete case.

We start by noting that the zero-integral conditions on $\omega$, $T$
and $S$ imply the Poincar{\'e} inequalities
\begin{equation}\label{q:cpoi}
   |\omega|_{L^2(\Dom)}^2 \le \cpoi\,|\gb\omega|_{L^2(\Dom)}^2,
\end{equation}
as well as the equivalence of the norms
\begin{equation}\label{q:normeq}
   |\omega|_{H^1(\Dom)}^{} \le c\,|\gb\omega|_{L^2(\Dom)}^{},
\end{equation}
with analogous inequalities for $T$ and $S$.
The boundary condition $\psi=0$ implies that \eqref{q:cpoi}--\eqref{q:normeq}
also hold for $\psi$, while an elliptic regularity estimate
\cite[Cor.~8.7]{gilbarg-trudinger:epde} implies that
\begin{equation}\label{q:wpoi}
   |\gb\psi|_{L^2(\Dom)}^2 \le \cpoi\,|\omega|_{L^2(\Dom)}^2.
\end{equation}
Following the argument in \cite{gtw3:dodu}, this also holds
for functions, such as our $T$ and $S$, with zero integrals in $\Dom$.

Let $\Omega$ be an $H^2$ extension of $\flu$ to $\bar\Dom$ (further
requirements will be imposed below) and let $\wh:=\omega-\Omega$;
we also define $\Delta\psih:=\wh$ and $\Delta\Psi:=\Omega$ with
homogeneous boundary conditions.
Now $\wh$ satisfies the homogeneous boundary conditions $\wh=0$ on $\dy\Dom$,
and thus the Poincar{\'e} inequality \eqref{q:cpoi}--\eqref{q:normeq}.
Furthermore, let $\tq\in \dot H^2(\Dom)$ be such that $\dy_z\tq=\flt$
on $\dy\Dom$ (with other constraints to be imposed below) and
let $\Th:=T-\tq$;
analogously for $\sq$ and $\Sh:=S-\sq$.
We note that since both $\Th$ and $\Sh$ have zero integrals over $\Dom$,
they satisfy the Poincar{\'e} inequality \eqref{q:cpoi}--\eqref{q:normeq}.

We start with weak solutions of \eqref{q:dudt}.
For conciseness, unadorned norms and inner products are understood to be
$L^2(\Dom)$, $|\cdot|:=|\cdot|_{L^2(\Dom)}^{}$ and
$(\cdot,\cdot):=(\cdot,\cdot)_{L^2(\Dom)}^{}$.
With $\wh$, $\Th$ and $\Sh$ as defined above, we have
\begin{equation}\label{q:duhdt}\begin{aligned}
   &\dy_t\wh + \dy(\Psi+\psih,\Omega+\wh) = \ppr\bigl\{\Delta\wh + \Delta\Omega + \dy_x\tq + \dy_x\Th - \dy_x\sq - \dy_x\Sh\bigr\}\\
   &\dy_t\Th + \dy(\Psi+\psih,\tq+\Th) = \Delta\tq + \Delta\Th\\
   &\dy_t\Sh + \dy(\Psi+\psih,\sq+\Sh) = \pts\,(\Delta\sq + \Delta\Sh).
\end{aligned}\end{equation}
On a fixed time interval $[0,T_*)$, a weak solution of \eqref{q:duhdt} are
\begin{equation}\begin{aligned}
   &\wh \in C^0(0,T_*;L^2(\Dom)) \cap L^2(0,T_*;H^1_0(\Dom))\\
   &\Th \in C^0(0,T_*;L^2(\Dom)) \cap L^2(0,T_*;H^1(\Dom))\\
   &\Sh \in C^0(0,T_*;L^2(\Dom)) \cap L^2(0,T_*;H^1(\Dom))
\end{aligned}\end{equation}
such that, for all $\tilde\omega\in H^1_0(\Dom)$, $\tilde T$,
$\tilde S\in H^1(\Dom)$, the following holds in the distributional sense,
\begin{equation}\begin{aligned}
   &\ddt{\;}(\wh,\tilde\omega)
   + (\dy(\Psi+\psih,\Omega+\wh),\tilde\omega)\\
   &\qquad+ \ppr\,\bigl\{ (\gb\Omega+\gb\wh,\gb\tilde\omega)
        - (\dy_x\tq+\dy_x\Th,\tilde\omega) + (\dy_x\sq+\dy_x\Sh,\tilde\omega) \bigr\} = 0\\
   &\ddt{\;}(\Th,\tilde T) + (\dy(\Psi+\psih,\tq+\Th),\tilde T)
	+ (\gb\Th,\gb\tilde T) - (\Delta\tq,\tilde T) = 0\\
   &\ddt{\;}(\Sh,\tilde S) + (\dy(\Psi+\psih,\sq+\Sh),\tilde S)
	+ \pts\,(\gb\Sh,\gb\tilde S) - \pts\,(\Delta\sq,\tilde S) = 0.
\end{aligned}\end{equation}
The existence of such solutions can be obtained by standard methods,
so we do not do so explicitly here.

Next, we derive $L^2$ inequalities for $T$, $S$ and $\omega$.
Multiplying (\ref{q:duhdt}a)
by $\wh$ in $L^2(\Dom)$ and noting that $(\dy(\psi,\wh),\wh)=0$, we find
\begin{equation}\begin{aligned}
   \frac12\ddt{\;}|\wh|^2 + \ppr\,|\gb\wh|^2
	&= -(\dy(\Psi,\Omega),\wh) - (\dy(\psih,\Omega),\wh)\\
	&\qquad {}+ \ppr\bigl\{(\Delta\Omega,\wh) + (\dy_xT,\wh) - (\dy_xS,\wh)\bigr\}.
\end{aligned}\end{equation}
We bound the rhs as
\begin{align*}
   &\bigl|(\Delta\Omega,\wh)\bigr| = |\gb\Omega|\,|\gb\wh|
	\le \sfrac18\,|\gb\wh|^2 + 2\,|\gb\Omega|^2\\
   &\bigl|(\dy_xT,\wh)\bigr| = |\dy_x\wh|\,|T|
	\le \sfrac18\,|\gb\wh|^2 + 2\,|T|^2
	\le \sfrac18\,|\gb\wh|^2 + 4\cpoi|\gb\Th|^2 + 4\,|\tq|^2\\
   &\bigl|(\dy_xS,\wh)\bigr| = |\dy_x\wh|\,|S|
	\le \sfrac18\,|\gb\wh|^2 + 2\,|S|^2
	\le \sfrac18\,|\gb\wh|^2 + 4\cpoi|\gb\Sh|^2 + 4\,|\sq|^2.
\end{align*}
and the ``nonlinear'' terms as
\begin{equation}\begin{aligned}
   \bigl|(\dy(\psih,\wh),\Omega)\bigr|
	&\le c\,|\gb\psih|_{L^\infty}^{}|\gb\wh|_{L^2}^{}|\Omega|_{L^2}^{}
	\le \frac{c_1}2|\Omega|_{L^2}^{}|\gb\wh|^2\\
   \bigl|(\dy(\Psi,\wh),\Omega)\bigr|
	&\le c\,|\gb\Psi|_{L^\infty}^{}|\gb\wh|_{L^2}^{}|\Omega|_{L^2}^{}
	\le \frac\ppr8\,|\gb\wh|^2 + \frac{c}\ppr\,|\gb\Psi|_{L^\infty}^2|\Omega|^2.
\end{aligned}\end{equation}
This brings us to
\begin{equation}\label{q:icl2w}\begin{aligned}
   \ddt{\;}|\wh|^2 + (\ppr-c_1|\Omega|)|\gb\wh|^2
	&\le 4\ppr\cpoi(|\gb\Th|^2 + |\gb\Sh|^2)\\
	&\quad+ \frac{c}\ppr\,|\gb\Psi|_{L^\infty}^2|\Omega|^2
	+ 4\ppr\,(|\gb\Omega|^2 + |\tq|^2 + |\sq|^2).
\end{aligned}\end{equation}

Now for $\Sh$, we multiply (\ref{q:dudt}c), or equivalently,
\begin{equation}\label{q:dshdt}\begin{aligned}
   \dy_t\Sh + \dy(\psi,\Sh+\sq) = \pts\,(\Delta\Sh + \Delta\sq),
\end{aligned}\end{equation}
by $\Sh$ in $L^2(\Dom)$ and use $(\dy(\psi,\Sh),\Sh)=0$ to find
\begin{equation}\begin{aligned}
   \frac12\ddt{\;}|\Sh|^2 + \pts\,|\gb\Sh|^2
	= -(\dy(\Psi,\sq),\Sh) - (\dy(\psih,\sq),\Sh) + \pts\,(\Delta\sq,\Sh).
\end{aligned}\end{equation}
The last term on the rhs requires some care,
\begin{equation}\label{q:BCst}\begin{aligned}
   \bigl|(\Delta\sq,\Sh)\bigr|
	&\le \bigl|(\fls,\Sh)_{L^2(\dy\Dom)}^{}\bigr| + \bigl|(\gb\sq,\gb\Sh)\bigr|\\
	&\le c\,|\fls|_{H^{-1/2}(\dy\Dom)}^{}|\Sh|_{H^{1/2}(\dy\Dom)}^{} + |\gb\fls|\,|\gb\Sh|\\
	&\le \sfrac18\,|\gb\Sh|^2 + c\,(|\gb\sq|^2 + |\fls|_{H^{-1/2}(\dy\Dom)}^2)\\
	&\le \sfrac18\,|\gb\Sh|^2 + c\,|\gb\sq|^2
\end{aligned}\end{equation}
where we have used the trace theorem for the second and last inequalities.
We note that $|\gb\sq|_{L^2(\Dom)}^{}$ ultimately depends on
$|\fls|_{H^{-1/2}(\dy\Dom)}^{}$ plus the constraint \eqref{q:qc} below.
Bounding the ``nonlinear'' terms as
\begin{equation}\begin{aligned}
   \bigl|(\dy(\Psi,\Sh),\sq)\bigr|
	&\le c\,|\gb\Psi|_{L^\infty}^{}|\gb\Sh|_{L^2}^{}|\sq|_{L^2}^{}
	\le \frac\pts8\,|\gb\Sh|^2 + \frac{c}\pts\,|\gb\Psi|_{L^\infty}^2|\sq|^2\\
   \bigl|(\dy(\psih,\Sh),\sq)\bigr|
	&\le c\,|\gb\psih|_{L^\infty}^{}|\gb\Sh|_{L^2}^{}|\sq|_{L^2}^{}
	\le \frac\pts8\,|\gb\Sh|^2 + \frac{c}\pts\,|\gb\wh|^2|\sq|^2,
\end{aligned}\end{equation}
we arrive at
\begin{equation}\label{q:icl2s}\begin{aligned}
   \ddt{\;}|\Sh|^2 + \pts\,|\gb\Sh|^2
	\le \frac{c}\pts\,|\gb\wh|^2|\sq|^2
	+ \frac{c}\pts\,|\gb\Psi|_{L^\infty}^2|\sq|^2
	+ c\pts\,|\gb\sq|^2.
\end{aligned}\end{equation}
Analogously, we have for $\Th$,
\begin{equation}\label{q:icl2t}
   \ddt{\;}|\Th|^2 + |\gb\Th|^2
	\le c\,|\gb\wh|^2|\tq|^2
	+ c\,|\gb\Psi|_{L^\infty}^2|\tq|^2 + c\,|\gb\tq|^2.
\end{equation}

Adding $8\ppr\cpoi$ times \eqref{q:icl2t} and
$8\ppr\cpoi/\pts$ times \eqref{q:icl2s} to \eqref{q:icl2w}, we find
\begin{equation}\label{q:aux00}\begin{aligned}
   \ddt{\;}\Bigl(|\wh|^2 &+ 8\ppr\cpoi|\Th|^2 + \frac{8\ppr\cpoi}\pts|\Sh|^2\Bigr)
	+ 4\ppr\cpoi\,(|\gb\Th|^2 + |\gb\Sh|^2)\\
	&\qquad\qquad+ \Bigl(\ppr-c_1|\Omega|-c_2\ppr|\tq|^2-\frac{c_2\ppr}{\pts^2}|\sq|^2\Bigl)|\gb\wh|^2\\
	&\le c\ppr\,|\gb\Psi|_{L^\infty}^2\bigl(|\Omega|^2/\ppr^2 + |\tq|^2 + |\sq|^2/\pts^2\bigr)\\
	&\qquad\qquad+ c\ppr\,(|\gb\Omega|^2 + |\gb\tq|^2 + |\gb\sq|^2).
\end{aligned}\end{equation}
If we now choose $\Omega$, $\tq$ and $\sq$ such that
\begin{equation}\label{q:qc}
   |\Omega|_{L^2}^{} \le \ppr/(8c_1),\quad
   |\tq|_{L^2}^2 \le 1/(8c_2)
   \aand
   |\sq|_{L^2}^2 \le \pts^2/(8c_2),
\end{equation}
(given the BC \eqref{q:BC},
this can always be done at the price of making $\gb\Omega$, $\gb\tq$ and
$\gb\sq$ large) we obtain the differential inequality
\begin{equation}\label{q:icl2}
   \ddt{\;}\Bigl(|\wh|^2 + 8\ppr\cpoi|\Th|^2 + \frac{8\ppr\cpoi}\pts|\Sh|^2\Bigr)
	+ \frac\ppr2|\gb\wh|^2 + 4\ppr\cpoi\,(|\gb\Th|^2 + |\gb\Sh|^2)
	\le \|F\|^2,
\end{equation}
with $\|F\|^2$ denoting the purely ``forcing'' terms on the rhs
of \eqref{q:aux00}.
Integrating this using the Gronwall lemma, we obtain the uniform bounds, with
$|\Uh|^2=|\wh|^2+8\ppr\cpoi|\Th|^2+8\ppr\cpoi|\Sh|^2/\pts$,
\begin{equation}\label{q:bdcl2}\begin{aligned}
   &|\Uh(t)|^2 \le \ex^{-\lambda t}|\Uh(0)|^2 + \|F\|^2/\lambda\\
   &c_3\ppr\int_t^{t+1} \bigl\{|\gb\wh|^2 + |\gb\Th|^2 + |\gb\Sh|^2\bigr\}(t') \dt'
	\le \ex^{-\lambda t}|U(0)|^2 + (1+1/\lambda)\,\|F\|^2
\end{aligned}\end{equation}
valid for all $t\ge0$, for some $\lambda(\parm)>0$.
It is clear from (\ref{q:bdcl2}a) that we have an absorbing ball,
i.e.\ $|U(t)|^2 \le M_0(\fla;\parm)$ for all $t\ge t_0(|U(0)|;\parm)$.

\medskip
On to $H^1$, we multiply (\ref{q:duhdt}a) by $-\Delta\wh$ in $L^2$ to find
\begin{equation}\begin{aligned}
   \frac12\ddt{\;}|\gb\wh|^2 &+ \ppr\,|\Delta\wh|^2
	= -(\dy(\gb\psi,\wh),\gb\wh) + (\dy(\psi,\Omega),\Delta\wh)\\
	&- \ppr\,(\Delta\Omega,\Delta\wh) - \ppr\,(\dy_x T,\Delta\wh) + \ppr\,(\dy_x S,\Delta\wh).
\end{aligned}\end{equation}
Bounding the linear terms in the obvious way, and the nonlinear terms as
\begin{align*}
   \bigl|(\dy(\gb\psi,\wh),\gb\wh)\bigr|
	&\le c\,|\gb\wh|_{L^4}^2|\gb^2\psi|_{L^2}^{}
	\le c\,|\gb\wh|\,|\Delta\wh|\,|\Delta\psi|\\
	&\le \frac\ppr8\,|\Delta\wh|^2 + \frac{c}\ppr\,|\gb\wh|^2(|\wh|^2 + |\Omega|^2)\\
   \bigl|(\dy(\psi,\Omega),\Delta\wh)\bigr|
	&\le \frac\ppr8\,|\Delta\wh|^2 + \frac{c}\ppr\,|\gb\Omega|^2\bigl(|\gb\wh|^2 + |\gb\Psi|_{L^\infty}^2\bigr),
\end{align*}
we find
\begin{equation}\label{q:ich1w}\begin{aligned}
   \ddt{\;}|\gb\wh|^2 &+ \ppr\,|\Delta\wh|^2
	\le \frac{c}\ppr\,|\gb\wh|^2(|\wh|^2 + |\Omega|^2 + |\gb\Omega|^2)
	+ \frac{c}\ppr\,|\gb\Psi|_{L^\infty}^2|\gb\Omega|^2\\
	&+ 8\ppr\,\bigl(|\gb\Th|^2 + |\gb\Sh|^2 + |\gb\tq|^2 + |\gb\sq|^2 + |\Delta\Omega|^2\bigr).
\end{aligned}\end{equation}
Since $\wh$, $\Th$ and $\Sh$ have been bounded uniformly in $L_{t,1}^2H_x^1$
in (\ref{q:bdcl2}b), we can integrate \eqref{q:ich1w} using the uniform
Gronwall lemma to obtain a uniform bound for $|\gb\wh|^2$,
\begin{equation}\label{q:bdch1w}
   |\gb\wh(t)|^2 \le M_1(\cdots)
   \qquad\textrm{and}\qquad
   \int_t^{t+1} |\Delta\wh(t')|^2 \dt' \le \mt_1(\cdots).
\end{equation}

Similarly, multiplying \eqref{q:dshdt} by $-\Delta\Sh$ in $L^2$, we find
\begin{equation}\begin{aligned}
   \frac12\ddt{\;}|\gb\Sh|^2 + \pts\,|\Delta\Sh|^2
	= &- \pts\,(\Delta\sq,\Delta\Sh)\\
	&- (\dy(\gb\psi,\Sh),\gb\Sh) + (\dy(\psi,\sq),\Delta\Sh).
\end{aligned}\end{equation}
Bounding as we did for $\wh$, we arrive at
\begin{equation}\begin{aligned}
   \ddt{\;}|\gb\Sh|^2 + \pts\,|\Delta\Sh|^2
	&\le 8\pts\,|\Delta\sq|^2\\
	&+ \frac{c}\pts\,|\gb\Sh|^2(|\wh|^2 + |\Omega|^2)
	+ \frac{c}\pts\,|\gb\sq|^2\bigl(|\gb\wh|^2 + |\gb\Psi|_{L^\infty}^2\bigr),
\end{aligned}\end{equation}
which can be integrated using the uniform Gronwall lemma to obtain
\begin{equation}\label{q:bdch1s}
   |\gb\Sh(t)|^2 \le M_1(\cdots)
   \qquad\textrm{and}\qquad
   \int_t^{t+1} |\Delta\Sh(t')|^2 \dt' \le \mt_1(\cdots).
\end{equation}
Obviously one has the analogous bound for $\Th$,
\begin{equation}\label{q:bdch1t}
   |\gb\Th(t)|^2 \le M_1(\cdots)
   \qquad\textrm{and}\qquad
   \int_t^{t+1} |\Delta\Th(t')|^2 \dt' \le \mt_1(\cdots).
\end{equation}


\medskip
These bounds allow us to conclude \cite{temam:iddsmp} the existence
of a global attractor $\Attr$ and of an invariant measure $\mu$
supported on $\Attr$.
Given a continuous functional $\Phi$, its long-time average satisfies
\begin{equation}
   \LIM_{t\to\infty} \frac1t \int_0^t \Phi(\dds(t)U_0) \dt
	= \int_H \Phi(U) \dmu(U)
\end{equation}
where $U(t)=\dds(t)U_0$ is the solution of \eqref{q:dudt} with initial
data $U_0$.
It is known that $\Attr$ is unique while $\mu$ may depend on
the initial data $U_0$ and the definition of the generalised limit $\LIM$.

Due to the boundary conditions, one cannot simply multiply by $\Delta^2\wh$,
etc., to obtain a bound in $H^2$, but following \cite[\S6.2]{temam:nsenfa},
one takes time derivative of (\ref{q:dudt}a) and uses the resulting bound
on $|\dy_t\omega|$ to bound $|\Delta\omega|$, etc.
We shall not do this explicitly here, although similar ideas are used for
the discrete case below (proof of Theorem~\ref{t:h2}).


\section{Numerical scheme: boundedness}\label{s:disc}

Fixing a timestep $k>0$,
we discretise the system \eqref{q:dudt} in time by the following two-step
explicit--implicit scheme,
\begin{equation}\label{q:ei}\begin{aligned}
   &\frac{3\omega^{n+1}-4\omega^n+\omega^{n-1}}{2k}
	+ \dy(2\psi^n-\psi^{n-1},2\omega^n-\omega^{n-1})\\
	&\hbox to132pt{}= \ppr\bigl\{\Delta\omega^{n+1} + \dy_x T^{n+1} - \dy_x S^{n+1}\bigr\}\\
   &\frac{3T^{n+1}-4T^n+T^{n-1}}{2k}
	+ \dy(2\psi^n-\psi^{n-1},2T^n-T^{n-1})
	= \Delta T^{n+1}\\
   &\frac{3S^{n+1}-4S^n+S^{n-1}}{2k}
	+ \dy(2\psi^n-\psi^{n-1},2S^n-S^{n-1})
	= \pts\Delta S^{n+1},
\end{aligned}\end{equation}
plus the boundary conditions \eqref{q:BC}.
Writing $U^n=(\omega^n,T^n,S^n)$,
we assume that the second initial data $U^1$ has been obtained from
$U^0$ using some reasonable one-step method, but all we shall need
for what follows is that $U^1\in H^1(\Dom)$.
The time derivative term is that of the backward differentiation
formula (BDF) and the explicit nonlinear term is sometimes known
as a ``one-leg method'' \cite[(V.6.6)]{hairer-wanner:sode2}.
This results in a method that is essentially explicit yet
second order in time, and as we shall see below, preserves
the important invariants of the continuous system.

Subject to some restrictions on the timestep $k$, we can obtain
uniform bounds and absorbing balls for the solution of the discrete
system analogous to those of the continuous system.
Our first result is the following:

\begin{theorem}\label{t:h1}
With $\ff\in H^{3/2}(\dy\Dom)$, the scheme \eqref{q:ei} defines
a discrete dynamical system in $H^1(\Dom)\times H^1(\Dom)$.
Assuming $U^0$, $U^1\in H^1(\Dom)$ and the timestep restriction
given in \eqref{q:k1} below,
\begin{equation}\label{q:dt}
   k \le k_1(|U^0|_{H^1}^{},|U^1|_{H^1}^{};|\ff|_{H^{1/2}(\dy\Dom)}^{},\parm),
\end{equation}
the following bounds hold
\begin{align}
   &|U^n|_{L^2}^2 \le 40\,\ex^{-\nu nk/4}\bigl(|U^0|_{L^2}^2 + |U^1|_{L^2}^2\bigr)
	+ M_0(|\ff|_{H^{1/2}(\dy\Dom)}^{};\parm)\notag\\
	&\hbox to60pt{}+ c(|\ff|_{H^{-1/2}(\dy\Dom)}^{};\parm)k\,\ex^{-\nu nk/4}\bigl(|U^0|_{H^1}^2 + |U^1|_{H^1}^2\bigr),\label{q:l2}\\
   &|U^n|_{H^1}^2 \le N_1(nk;|U^0|_{H^1}^{},|U^1|_{H^1}^{},|\ff|_{H^{1/2}(\dy\Dom)}^{},\parm) + M_1(|\ff|_{H^{3/2}(\dy\Dom)}^{};\parm), \label{q:h1}
\end{align}
where $\nu(\parm)>0$ and
$N_1(t;\cdots)=0$ for $t\ge t_1(|U^0|_{H^1}^{},|U^1|_{H^1}^{};\ff,\parm)$.
\end{theorem}

We note that the last term in \eqref{q:l2} has no analogue in the
continuous case;
we believe this is an artefact of our proof, but have not been able
to circumvent it.
Unlike in \cite{wxm:12}, $H^2$ bounds do not follow as readily due to
the boundary conditions, so we proceed by first deriving bounds for
$|U^{n+1}\!-U^n|$, using an approach inspired by \cite[\S6.2]{temam:nsenfa}.
We state our result without the transient terms:

\begin{theorem}\label{t:h2}
Assume the hypotheses of Theorem~\ref{t:h1}.
Then for sufficiently large $nk$, one has
\begin{equation}\label{q:bddU}
   |\omega^{n+1}\!-\omega^n|^2 + |T^{n+1}\!-T^n|^2 + |S^{n+1}\!-S^n|^2
	\le k^2 M_\delta(|Q|_{H^{3/2}(\dy\Dom)}^{};\parm).
\end{equation}
Furthermore, for large $nk$ the solution is bounded in $H^2$ as
\begin{equation}\label{q:h2}
   |\Delta\omega^n|^2 + |\Delta T^n|^2 + |\Delta S^n|^2 \le M_2(|Q|_{H^{3/2}(\dy\Dom)}^{};\parm).
\end{equation}
\end{theorem}

We remark that these difference and $H^2$ bounds require no additional
hypotheses on $Q$, indicating that Theorem~\ref{t:h1} may be sub-optimal.
We also note that using the same method (and one more derivative on $Q$)
one could bound $|U^{n+1}-U^n|_{H^1}^{}$ and $|U^n|_{H^3}^{}$,
although we will not need these results here.

Following the approach of \cite{wxm:12}, these uniform bounds
(along with the uniform convergence results that follow from them)
then give us the convergence of long-time statistical properties of the
discrete dynamical system \eqref{q:ei} to those of the continuous
system \eqref{q:dudt}.



\begin{proof}[Proof of Theorem~\ref{t:h1}]
Central to our approach is the idea of $G$-stability for multistep methods
\cite[\S V.6]{hairer-wanner:sode2}.
First, for $f$, $g\in L^2(\Dom)$ and $\nu k\in[0,1]$, we define the norm
\begin{equation}\label{q:gndef}
   \gn{f,g}_{\nu k}^2
	= \frac{|f|_{L^2}^2}2 + \frac{5+\nu k}2\,|g|_{L^2}^2 - 2(f,g)_{L^2}^{}.
\end{equation}
Note that our notation is slightly different
from that in \cite{hill-suli:00,wxm:12}.
Since both eigenvalues of the quadratic form are finite and positive
for all $\nu k\in[0,1]$, this norm is equivalent to the $L^2$ norm,
i.e.\ there exist positive constants $c_+$ and $c_-$,
independent of $\nu k\in[0,1]$, such that
\begin{equation}\label{q:equivn}
   c_-(|f|_{L^2}^2 + |g|_{L^2}^2)
	\le \gn{f,g}_{\nu k}^2
	\le c_+(|f|_{L^2}^2 + |g|_{L^2}^2)
\end{equation}
for all $f$, $g\in L^2(\Dom)$;
computing explicitly, we find
\begin{equation}\label{q:c+-}
   c_- = \frac{6-\sqrt{32}}4
   \qquad\textrm{and}\qquad
   c_+ = \frac{7+\sqrt{41}}4.
\end{equation}
As in \cite{wxm:12}, an important tool for our estimates is an identity
first introduced in \cite{hairer-wanner:sode2} for $\nu k=0$;
the following form can be found in \cite[proof of Lemma~6.1]{hill-suli:00}:
for $f$, $g$, $h\in L^2(\Dom)$ and $\nu k\in[0,1]$,
\begin{equation}\label{q:hs00}\begin{aligned}
   (3h-4g+f,h)_{L^2}^{} &+ \nu k\,|h|_{L^2}^2\\
	&= \gn{g,h}_{\nu k}^2 - \frac1{1+\nu k}\gn{f,g}_{\nu k}^2
	+ \frac{|f-2g+(1+\nu k)h|_{L^2}^2}{2(1+\nu k)}.
\end{aligned}\end{equation}

As usual, $c$ denotes generic constants which may take different values
each time it appears.
Numbered constants such as $\cpoi$ have fixed values;
they are independent of the parameters $\ppr$ and $\pts$ unless
noted explicitly.

\medskip
The fact that \eqref{q:ei} forms a discrete dynamical system in $H^1\times H^1$
can be seen by writing
\begin{equation}
   (3-2k\Delta) T^{n+1} = 4T^n - T^{n-1} - 2k\,\dy(2\psi^n-\psi^{n-1},2T^n-T^{n-1})
\end{equation}
and inverting: given $U^{n-1}$ and $U^n\in H^1(\Dom)$, the Jacobian is
in $H^{-1}$, which, with the Neumann BC $\dy_z T^{n+1}=\flt\in H^{1/2}(\dy\Dom)$,
gives $T^{n+1}\in H^1$.
Similarly for $S^{n+1}$ and, since now $T^{n+1}$, $S^{n+1}\in H^1$
and $\omega^{n+1}=\flu\in H^{1/2}(\dy\Dom)$, for $\omega^{n+1}$.
Therefore $(U^{n-1},U^n)\in H^1\times H^1$ maps to $(U^n,U^{n+1})\in H^1\times H^1$.

Let $\wh^n:=\omega^n-\Omega$, $\Th^n:=T^n-\tq$ and $\Sh^n:=S^n-\sq$ be defined
as in the continuous case, i.e.\ $\Omega$, $\tq$, $\sq\in H^2(\Dom)$
satisfying the boundary conditions
$\Omega=\flu$, $\dy_z\tq=\flt$ and $\dy_z\sq=\fls$,
and the constraint \eqref{q:qd}, which is essentially \eqref{q:qc}.
The scheme \eqref{q:ei} then implies
\begin{equation}\label{q:eih}\begin{aligned}
   &\frac{3\wh^{n+1}\!-4\wh^n\!+\wh^{n-1}}{2k}
	+ \dy(2\psi^n\!-\psi^{n-1},2\wh^n\!-\wh^{n-1}\!+\Omega)\\
	&\hbox to132pt{}= \ppr\bigl\{\Delta\wh^{n+1} + \Delta\Omega + \dy_x T^{n+1} - \dy_x S^{n+1}\bigr\}\\
   &\frac{3\Th^{n+1}\!-4\Th^n\!+\Th^{n-1}}{2k}
	+ \dy(2\psi^n\!-\psi^{n-1},2\Th^n\!-\Th^{n-1}\!+\tq)
	= \Delta \Th^{n+1} + \Delta\tq\\
   &\frac{3\Sh^{n+1}\!-4\Sh^n\!+\Sh^{n-1}}{2k}
	\!+ \dy(2\psi^n\!-\psi^{n-1}\!,2\Sh^n\!-\Sh^{n-1}\!\!+\sq)
	= \pts(\Delta \Sh^{n+1} \!+ \!\Delta\sq)
\end{aligned}\end{equation}
where we have kept some $\psi^n$, $T^n$ and $S^n$ for now.
We start by deriving difference inequalities for $\wh^n$, $\Th^n$ and $\Sh^n$.
In order to bound terms of the form
$|\gb\psih^n|_{L^\infty}^2\le c\,|\wh^n|_{H^{1/2}}^2$,
we assume for now the uniform bound
\begin{equation}\label{q:whalf}
   |\wh^n|_{H^{1/2}}^2 \le k^{-1/2}\mw(\cdots)
	\qquad\textrm{for all }n=0,1,2,\cdots
\end{equation}
where $\mw$ will be fixed in \eqref{q:whalf1} below.
We also assume for clarity that $k\le1$.

Multiplying (\ref{q:eih}a) by $2k\wh^{n+1}$ in $L^2(\Dom)$ and
using \eqref{q:hs00}, we find
\begin{equation}\label{q:c00}\begin{aligned}
   \gn{\wh^n,\wh^{n+1}}_{\nu k}^2 - \nu k\,|\wh^{n+1}|^2 + 2\ppr k\,|\gb\wh^{n+1}|^2
	+ \frac{|(1+\nu k)\wh^{n+1} \!- 2\wh^n \!+ \wh^{n-1}|^2}{2\,(1+\nu k)}\\
        = \frac{\gn{\wh^{n-1},\wh^n}_{\nu k}^2}{1\!+\!\nu k}
	- 2k\,(\dy(2\psi^n\!-\psi^{n-1},\wh^{n+1}),(1+\nu k)\wh^{n+1}\!-2\wh^n\!+\wh^{n-1})\\
	{}+ 2k\,(\dy(2\psih^n\!-\psih^{n-1},\wh^{n+1}),\Omega)
	+ 2k\,(\dy(\Psi,\wh^{n+1}),\Omega)\\
	{}+ 2\ppr k\,\bigl\{(\Delta\Omega,\wh^{n+1}) + (\wh^{n+1},\dy_x T^{n+1})
        - (\wh^{n+1},\dy_x S^{n+1})\bigr\}.
\end{aligned}\end{equation}
where $\nu>0$ will be set below.
We bound the last terms as in the continuous case,
\begin{align*}
   &2\,|(\Delta\Omega,\wh^{n+1})\bigr| \le \sfrac18\,|\gb\wh^{n+1}|^2 + 8\,|\gb\Omega|^2\\
   &2\,|(\dy_x T^{n+1},\wh^{n+1})| \le \sfrac18\,|\gb\wh^{n+1}|^2 + 16\cpoi\,|\gb\Th^{n+1}|^2 + 16\,|\tq|^2\\
   &2\,|(\dy_x S^{n+1},\wh^{n+1})| \le \sfrac18\,|\gb\wh^{n+1}|^2 + 16\cpoi\,|\gb\Sh^{n+1}|^2 + 16\,|\sq|^2\\
   &2\,\bigl|(\dy(\Psi,\wh^{n+1}),\Omega)| \le \frac\ppr8\,|\gb\wh^{n+1}|^2 + \frac{c}\ppr\,|\gb\Psi|_{L^\infty}^2|\Omega|^2,
\end{align*}
and the previous one as
\begin{equation}\begin{aligned}
   2\,|(\dy(2\psih^n-\psih^{n-1},\wh^{n+1}),\Omega)|
	&\le c\,|2\gb\psih^n-\gb\psih^{n-1}|_{L^\infty}^{}|\gb\wh^{n+1}|_{L^2}^{}|\Omega|_{L^2}^{}\\
	&\le \frac\ppr8\,|\gb\wh^{n+1}|^2 + \frac{c}\ppr\,(|\gb\wh^{n-1}|^2 + |\gb\wh^n|^2)|\Omega|^2.
\end{aligned}\end{equation}
Taking $\nu=\ppr/(8\cpoi)$ for now, we can bound the second term
in \eqref{q:c00} using the third.
Using \eqref{q:whalf}, we then bound the first nonlinear term as
\begin{equation}\label{q:c01}\begin{aligned}
   \!\!\!\!2\,|(\dy&(2\psi^n\!-\psi^{n-1},\wh^{n+1}),(1+\nu k)\wh^{n+1}\!-2\wh^n\!+\wh^{n-1})|\\
	&\le \frac\ppr8\,|\gb\wh^{n+1}|^2 + \frac{c}\ppr\,|2\gb\psi^n-\gb\psi^{n-1}|_{L^\infty}^2|(1+\nu k)\wh^{n+1}\!-2\wh^n\!+\wh^{n-1}|^2\\
	&\le \frac\ppr8\,|\gb\wh^{n+1}|^2 + c_3\,(k^{-1/2}\mw + |\gb\Psi|_{L^\infty}^2)\frac{|(1+\nu k)\wh^{n+1}\!-2\wh^n\!+\wh^{n-1}|^2}{4\ppr}.
\end{aligned}\end{equation}
Recalling that the validity of \eqref{q:equivn} and \eqref{q:c+-}
demands $k\le1/\nu$, which we henceforth assume,
we have $2(1+\nu k)\le4$.
This then implies that $k$ times the last term in \eqref{q:c01} can be
majorised by the fourth term in \eqref{q:c00} if $k$ is small enough that
\begin{equation}
   c_3k^{1/2}(\mw + |\gb\Psi|_{L^\infty}^2) \le \ppr.
\end{equation}
All this brings us to [cf.~\eqref{q:icl2w}]
\begin{equation}\label{q:idl2w}\begin{aligned}
   \!\!\!\!\!\gn{\wh^n,\wh^{n+1}}_{\nu k}^2 &+ \ppr k\,|\gb\wh^{n+1}|^2
	\le \frac{\gn{\wh^{n-1},\wh^n}_{\nu k}^2}{1+\nu k}\\
	&+ \frac{ck}\ppr\,(|\gb\wh^{n-1}|^2 + |\gb\wh^n|^2)|\Omega|^2
	+ {16\cpoi\ppr k}\,(|\gb\Th^{n+1}|^2 + |\gb\Sh^{n+1}|^2)\\
	&+ ck\,(|\gb\Psi|_{L^\infty}^2|\Omega|^2/\ppr + \ppr\,|\tq|^2 + \ppr\,|\sq|^2 + \ppr\,|\gb\Omega|^2).
\end{aligned}\end{equation}

For $\Sh^n$, we multiply (\ref{q:eih}c) by $2k\Sh^{n+1}$ in $L^2(\Dom)$
and use \eqref{q:hs00} to find
\begin{equation*}\begin{aligned}
   \gn{\Sh^n,\Sh^{n+1}}_{\nu k}^2 - \nu k\,|\Sh^{n+1}|^2 + 2\pts k\,|\gb\Sh^{n+1}|^2
	+ \frac{|(1+\nu k)\Sh^{n+1}\!-2\Sh^n\!+\Sh^{n-1}|^2}{2\,(1+\nu k)}\\
	= \frac{\gn{\Sh^{n-1},\Sh^n}_{\nu k}^2}{1+\nu k}
	- 2k\,(\dy(2\psi^n\!-\psi^{n-1},\Sh^{n+1}),(1+\nu k)\,\Sh^{n+1}\!-2\Sh^n\!+\Sh^{n-1})\\
	{}+ 2k\,(\dy(2\psih^n\!-\psih^{n-1},\Sh^{n+1}),\sq)
	+ 2k\,(\dy(\Psi,\Sh^{n+1}),\sq)
	+ 2\pts k\,(\Delta\sq,\Sh^{n+1}).
\end{aligned}\end{equation*}
Bounding the last term as in \eqref{q:BCst} and everything else as
with $\wh^n$, and taking (this also takes care of $\Th^n$ below)
\begin{align}
   &\nu = \min\{\ppr,\pts,1\}/(8\cpoi) \label{q:nu}\\
   &k \le \min\bigl\{\min\{\ppr^2,\pts^2,1\}/(c_3\mw + c_3|\gb\Psi|_{L^\infty}^2)^2,1/\nu\bigr\}, \label{q:k1}
\end{align}
we arrive at
\begin{equation}\label{q:idl2s}\begin{aligned}
   \!\!\gn{\Sh^n,\Sh^{n+1}}_{\nu k}^2 &+ \pts k\,|\gb\Sh^{n+1}|^2
	\le \frac{\gn{\Sh^{n-1}\!,\Sh^n}_{\nu k}^2}{1+\nu k}
	+ \frac{ck}\pts\,(|\gb\wh^{n-1}|^2 \!+ |\gb\wh^n|^2)|\sq|^2\\
	&+ \frac{ck}\pts\,|\gb\Psi|_{L^\infty}^2|\sq|^2
	+ c\pts k\,|\gb\sq|^2. 
\end{aligned}\end{equation}
Similarly, for $\Th^n$ we have
\begin{equation}\label{q:idl2t}\begin{aligned}
   \!\!\gn{\Th^n,\Th^{n+1}}_{\nu k}^2 &+ k\,|\gb\Th^{n+1}|^2
	\le \frac{\gn{\Th^{n-1},\Th^n}_{\nu k}^2}{1+\nu k}
	+ {ck}\,(|\gb\wh^{n-1}|^2 + |\gb\wh^n|^2)|\tq|^2\\
	&+ {ck}\,|\gb\Psi|_{L^\infty}^2|\tq|^2
	+ ck\,|\gb\tq|^2. 
\end{aligned}\end{equation}

Adding $16\ppr\cpoi$ times \eqref{q:idl2t} and
$16\ppr\cpoi/\pts$ times \eqref{q:idl2s} to \eqref{q:idl2w},
and writing
\begin{equation}
   \gn{\Uh^n,\Uh^{n+1}}_{\nu k}^2 := \gn{\wh^n,\wh^{n+1}}_{\nu k}^2
	+ 16\ppr\cpoi\gn{\Th^n,\Th^{n+1}}_{\nu k}^2
	+ 16\ppr\cpoi\gn{\Sh^n,\Sh^{n+1}}_{\nu k}^2/\pts,
\end{equation}
we have
\begin{equation}\label{q:idl2}\begin{aligned}
   \gn{\Uh^n,\Uh^{n+1}}_{\nu k}^2
	&+ \ppr k\,\bigl(|\gb\wh^{n+1}|^2 + 8\cpoi|\gb\Th^{n+1}|^2 + 8\cpoi|\gb\Sh^{n+1}|^2/\pts\bigr)\\
	&\le \frac{\gn{\Uh^{n-1},\Uh^n}_{\nu k}^2}{1+\nu k} + k\,\|F_1\|^2(|\gb\wh^{n-1}|^2 + |\gb\wh^n|^2) + k\,\|F_2\|^2
\end{aligned}\end{equation}
where
\begin{equation}\begin{aligned}
   &\|F_1\|^2 := c_4\ppr\,\bigl(|\Omega|^2/\ppr^2 + |\tq|^2 + |\sq|^2/\pts^2\bigr)\\
   &\|F_2\|^2 := |\gb\Psi|_{L^\infty}^2\|F_1\|^2
	+ c\ppr\,\bigl(|\gb\tq|^2 + |\gb\sq|^2 + |\gb\Omega|^2\bigr).
\end{aligned}\end{equation}

In order to integrate this difference inequality, we consider
a three-term recursion of the form
\begin{equation}\label{q:3tr}
   x_{n+1} + \mu y_{n+1} \le (1+\delta)^{-1} x_n + \eps y_n + \eps y_{n-1} + r_n.
\end{equation}
For $\mu>0$, $\delta\in(0,1]$ and $\eps\in(0,\mu/8]$, we have
\begin{equation}
   x_n + \mu y_n \le \frac{x_{n-m} + \mu y_{n-m}}{(1+\delta)^m}
	+ \frac{\eps\,y_{n-m-1}}{(1+\delta)^{m-1}}
	+ \sum\nolimits_{j=1}^m \frac{r_{n-j}}{(1+\delta)^{j-1}}
\end{equation}
(which follows readily by induction) and in particular
\begin{equation}\label{q:3tb}
   x_{n+1} + \mu y_{n+1} \le \frac{x_1 + \mu y_1}{(1+\delta)^n}
	+ \frac{\eps\,y_0}{(1+\delta)^{n-1}}
	+ \sum\nolimits_{j=1}^n \frac{r_j}{(1+\delta)^{n-j}}.
\end{equation}
In order to apply the bound \eqref{q:3tb} of \eqref{q:3tr} to \eqref{q:idl2},
we demand that $\Omega$, $\tq$ and $\sq$ be small enough that
\begin{equation}\label{q:qd}
   |\Omega|_{L^2}^2 \le \ppr^2/(32c_4),\qquad
   |\tq|_{L^2}^2 \le 1/(32c_4)
   \aand
   |\sq|_{L^2}^2 \le \pts^2/(32c_4).
\end{equation}
We note that, up to parameter-independent constants, these conditions
are identical to those in the continuous case \eqref{q:qc}.
Using the fact that $(1+x)^{-1} \le \exp(-x/2)$ for $x\in(0,1]$,
we integrate \eqref{q:idl2} to find a bound uniform in $t_n=nk$,
\begin{equation}\label{q:bdl2}\begin{aligned}
   \gn{\Uh^n,\Uh^{n+1}}_{\nu k}^2 &+ \ppr k\,|\gb\wh^{n+1}|^2\\
	&\le \ex^{-\nu nk/2}\bigl\{\gn{\Uh^0,\Uh^1}_{\nu k}^2 + \ppr k\,(|\gb\wh^0|^2 + |\gb\wh^1|^2)\bigr\}
	+ \frac2\nu\,\|F_2\|^2.
\end{aligned}\end{equation}
Using \eqref{q:equivn}--\eqref{q:c+-}, \eqref{q:l2} follows.

The hypothesis \eqref{q:whalf} can now be recovered by interpolation,
\begin{equation}\label{q:whalf1}\begin{aligned}
   |\wh^n|_{H^{1/2}}^2 &\le c\,|\wh^n|\,|\gb\wh^n|
	\le c\,\gn{\Uh^{n-1},\Uh^n}_{\nu k}^{}|\gb\wh^n|\\
	&\le c\,(\ppr k)^{-1/2}\bigl\{ \gn{\Uh^0,\Uh^1}_{\nu k}^2 + \ppr\,(|\gb\wh^0|^2 + |\gb\wh^1|^2) + 2\,\|F_2\|^2/\nu\bigr\}
\end{aligned}\end{equation}
and replacing $\gn{\Uh^0,\Uh^1}_{\nu k}^2$ by its sup over $\nu k\in(0,1]$.
Summing \eqref{q:idl2} and using \eqref{q:qd}, we find (discarding
terms on the lhs)
\begin{equation}\label{q:bdl2h1}\begin{aligned}
	k \sum\nolimits_{j=n+1}^{n+m} &\Bigl\{ \frac\ppr2\,|\gb\wh^j|^2 + 8\cpoi\,|\gb\Th^j|^2 + \frac{8\cpoi}\pts\,|\gb\Sh^j|^2\Bigr\}\\
	&\le \gn{\Uh^{n-1},\Uh^n}_{\nu k}^2 + 2k\,\|F_1\|^2(|\gb\wh^{n-1}|^2 + |\gb\wh^n|^2) + mk\,\|F_2\|^2.
\end{aligned}\end{equation}
From \eqref{q:bdl2} and \eqref{q:bdl2h1}, it is clear that there exists
a $t_0(|\gb U^0|,|\gb U^1|,Q;\pi)$ such that, whenever $nk\ge t_0$,
\begin{equation}\label{q:bdl2u}
   |\Uh^n|^2 \le M_0(Q;\pi)
   \aand
   k\,\nlsum_{j=n}^{n+\nk{1/k}}\,|\gb \Uh^j|^2 \le \mt_0(Q;\pi).
\end{equation}
We redefine $M_0$ and $\mt_0$ to bound $|U^n|^2$ and $\sum_j|\gb U^j|^2$
as well.

\medskip
On to $H^1$, we multiply (\ref{q:eih}a) by $-2k\Delta\wh^{n+1}$ in $L^2$ to get
\begin{equation}\begin{aligned}
   \qquad&\hskip-2em\gn{\gb\wh^n,\gb\wh^{n+1}}_{\nu k}^2 - \nu k\,|\gb\wh^{n+1}|^2
	+ \frac{|(1+\nu k)\gb\wh^{n+1}-2\gb\wh^n+\gb\wh^{n-1}|^2}{2\,(1+\nu k)}\\
	&= \frac{\gn{\gb\wh^{n-1},\gb\wh^n}_{\nu k}^2}{1+\nu k}
	- 2\ppr k\,|\Delta\wh^{n+1}|^2\\
	&- 2k\,(\dy(2\psi^n-\psi^{n-1},\gb\wh^{n+1}),(1+\nu k)\gb\wh^{n+1}-2\gb\wh^n+\gb\wh^{n-1})\\
	&- 2k\,(\dy(2\gb\psih^n-\gb\psih^{n-1},2\wh^n-\wh^{n-1}),\gb\wh^{n+1})\\
	&- 2k\,(\dy(\gb\Psi,2\wh^n-\wh^{n-1}),\gb\wh^{n+1})
	+ 2k\,(\dy(2\psi^n-\psi^{n-1},\Omega),\Delta\wh^{n+1})\\
	&+ 2\ppr k\,(\dy_xS^{n+1} - \dy_xT^{n+1} - \Delta\Omega,\Delta\wh^{n+1})
\end{aligned}\end{equation}
Labelling the ``nonlinear'' terms by \circled{1}, $\cdots$,\circled{4},
we bound them as
\begin{align*}
   \circled{1} &\le ck\,|2\gb\psi^n\!-\gb\psi^{n-1}|_{L^\infty}^{}|\gb^2\wh^{n+1}|_{L^2}^{}|(1+\nu k)\gb\wh^{n+1}\!-2\gb\wh^n\!+\gb\wh^{n-1}|_{L^2}^{}\\
	&\le \frac{\ppr k}8|\Delta\wh^{n+1}|^2
	+ \frac{c_3k^{1/2}}{4\ppr}\bigl(\mw + |\gb\Psi|_{L^\infty}^2\bigr)|\gb((1+\nu k)\wh^{n+1}\!-2\wh^n\!+\wh^{n-1})|^2\\
   \circled{2} &\le ck\,|2\wh^n-\wh^{n-1}|_{L^4}^{}|\gb^2\wh^{n+1}|_{L^2}^{}|2\wh^n-\wh^{n-1}|_{L^4}^{}\\
	&\le \frac{\ppr k}8\,|\Delta\wh^{n+1}|^2 + \frac{ck}\ppr\,|2\wh^n-\wh^{n-1}|^2|2\gb\wh^n-\gb\wh^{n-1}|^2\\
   \circled{3} &\le ck\,|\Omega|_{L^\infty}^{}|\gb^2\wh^{n+1}|_{L^2}^{}|2\wh^n-\wh^{n-1}|_{L^2}^{}\\
	&\le \frac{\ppr k}8\,|\Delta\wh^{n+1}|^2 + \frac{ck}\ppr\,|\Omega|_{L^\infty}^2|2\wh^n-\wh^{n-1}|^2\\
   \circled{4} &\le ck\,|2\gb\psi^n-\gb\psi^{n-1}|_{L^\infty}^{}|\gb\Omega|_{L^2}^{}|\Delta\wh^{n+1}|_{L^2}^{}\\
	&\le \frac{\ppr k}8\,|\Delta\wh^{n+1}|^2 + \frac{ck}\ppr\,|\gb\Omega|^2\bigl(|\gb\Psi|_{L^\infty}^2 + |2\gb\wh^n-\gb\wh^{n-1}|^2\bigr).
\end{align*}
Bounding the linear term in the obvious fashion and again using
\eqref{q:nu}--\eqref{q:k1}, we arrive at
\begin{equation}\label{q:idh1w}\begin{aligned}
   &\hskip-2em\gn{\gb\wh^n,\gb\wh^{n+1}}_{\nu k}^2 + \ppr k\,|\Delta\wh^{n+1}|^2\\
	&\le \gn{\gb\wh^{n-1},\gb\wh^n}_{\nu k}^2\bigl[1 \!+ {c\ppr^{-1}k}\,(M_0 \!+ |\gb\Omega|^2)\bigr]
	+ 8\ppr k\,\bigl(|\gb\Th^{n+1}|^2 \!+ |\gb\Sh^{n+1}|^2\bigr)\\
	&\quad+ {c\ppr^{-1}k}\,\bigl(M_0 |\Omega|_{L^\infty}^2 + |\gb\Omega|^2|\gb\Psi|_{L^\infty}^2\bigr)
	+ 8\ppr k\,\bigl(|\Delta\Omega|^2 + |\gb\tq|^2 + |\gb\sq|^2\bigr)
\end{aligned}\end{equation}
valid for large times $nk\ge t_0$.

Noting that, for $x_n\ge0$, $r_n\ge0$ and $b>0$,
\begin{equation}\label{q:gw1}
   x_{n+1} \le (1+b)\,x_n + r_n
   \qquad\Rightarrow\qquad
   x_{n+m} \le (1+b)^m\bigl(x_n + \tssum_{j=n}^{n+m-1}\,r_j\bigr),
\end{equation}
we can obtain a uniform $H^1$ bound from \eqref{q:bdl2u} and
\eqref{q:idh1w} as follows.
Borrowing an argument from \cite{coti-tone:12}, we conclude from
\eqref{q:bdl2u} that there exists an
$n_*\in\{n+\nk{1/k},\cdots,n+\nk{2/k}-1\}$
such that
\begin{equation}\label{q:ah100}
   |\gb\wh^{n_*}|^2 + |\gb\wh^{n_*+1}|^2 \le c\,\mt_0(Q;\parm)
	\>\Rightarrow\>
   \gn{\gb\wh^{n_*},\gb\wh^{n_*+1}}_{\nu k}^2 \le c_5\,\mt_0.
\end{equation}
(In other words, in any sequence of non-negative numbers, one can find
two consecutive terms whose sum is no greater than four times the average.)
Taking $n_*\in\{\Nk{t_0/k},\cdots,\Nk{(t_0+1)/k}-1\}$ and integrating
\eqref{q:idh1w} using \eqref{q:gw1} with $m=\nk{2/k}$ and \eqref{q:bdl2u}
to bound the $|\gb\Th^n|^2$ and $|\gb\Sh^n|^2$ on the rhs, we find
\begin{equation}\label{q:bdh1w}
   \gn{\gb\wh^n,\gb\wh^{n+1}}_{\nu k}^2 \le M_1(Q;\pi)
\end{equation}
for all $n\in\{n_*,\cdots,n_*+\nk{2/k}-1\}$.
We then find a $n_{**}\in\{n_*+\nk{1/k},\cdots,n_*+\nk{2/k}-1\}$ that satisfies
\eqref{q:ah100} and repeat the argument to find that \eqref{q:bdh1w} also
holds for all $n\in\{n_{**},\cdots,n_{**}+\nk{2/k}-1\}$.
Since $n_{**}\ge n_*+\nk{1/k}$, with each iteration we increase the time of
validity of \eqref{q:bdh1w} by at least 1 using no further assumptions,
implying that \eqref{q:bdh1w} in fact holds for all $n\ge n_*$,
i.e.\ whenever $nk\ge t_0+1$.

Similarly for $\Sh^n$, we multiply (\ref{q:eih}c) by $-2k\Delta\Sh^{n+1}$
in $L^2$ to find after a similar computation
\begin{equation}\label{q:idh1s}\begin{aligned}
   \gn{\gb\Sh^n,\gb\Sh^{n+1}}_{\nu k}^2 &+ \pts k\,|\Delta\Sh^{n+1}|^2
	\le \gn{\gb\Sh^{n-1},\Sh^n}_{\nu k}^2 \bigl(1 + ck\pts^{-1}M_0\bigr)\\
	&+ \frac{ck}\pts\,(M_0 + |\gb\sq|^2)\bigl(|\gb\Psi|_{L^\infty}^2 + |\gb\wh^{n-1}|^2 + |\gb\wh^n|^2\bigr)\\
	&+ \frac{ck}\pts\,M_0|\Omega|_{L^\infty}^2
	+ 8\pts k\,|\Delta\sq|^2.
\end{aligned}\end{equation}
Arguing as we did with $\wh^n$, we conclude that (redefining $M_1$ as needed)
one has
\begin{equation}
   \gn{\gb\Sh^n,\gb\Sh^{n+1}}_{\nu k}^2 \le M_1(Q;\pi)
	\qquad\textrm{whenever } nk \ge t_0+1.
\end{equation}
Obviously the same bound applies to $\Th^n$,
\begin{equation}
   \gn{\gb\Th^n,\gb\Th^{n+1}}_{\nu k}^2 \le M_1(Q;\pi)
	\qquad\textrm{whenever } nk \ge t_0+1.
\end{equation}
As with $M_0$, we redefine $M_1$ to bound
$\gn{\gb\omega^n,\gb\omega^{n+1}}_{\nu k}^2$, etc., as well as
$\gn{\gb\wh^n,\gb\wh^{n+1}}_{\nu k}^2$.
\end{proof}


\hbox to\hsize{\qquad\hrulefill\qquad}\medskip

\begin{proof}[Proof of Theorem~\ref{t:h2}]
Let $\delta U^n:=U^n-U^{n-1}=\hat U^n-\hat U^{n-1}$.
We first prove that $|\delta U^n|^2\le kM$ for all large $n$,
and then use this result to prove \eqref{q:bddU}.

\medskip
Writing $3\omega^{n+1}-4\omega^n+\omega^{n-1}=3\delta\omega^{n+1}-\delta\omega^n$
and using the identity
\begin{equation}\label{q:idd1}
  2\,(3\delta\omega^{n+1}-\delta\omega^n,\delta\omega^{n+1})
	= 3\,|\delta\omega^{n+1}|^2 - \sfrac13\,|\delta\omega^n|^2 + \sfrac13\,|3\delta\omega^{n+1}-\delta\omega^n|^2,
\end{equation}
we multiply (\ref{q:ei}a) by $4k\delta\omega^{n+1}$,
\begin{equation}\begin{aligned}
   3\,|\delta\omega^{n+1}|^2 &+ \sfrac13\,|3\delta\omega^{n+1}-\delta\omega^n|^2
	= \sfrac13\,|\delta\omega^n|^2\\
	&+ 4\ppr k\,(\Delta\omega^{n+1},\delta\omega^{n+1})
	+ 4\ppr k\,(\dy_x T^{n+1}-\dy_x S^{n+1},\delta\omega^{n+1})\\
	&- 4k\,(\dy(2\psi^n-\psi^{n-1},2\omega^n-\omega^{n-1}),\delta\omega^{n+1}).
\end{aligned}\end{equation}
For the dissipative term, we integrate by parts using the fact that
$\delta\omega^{n+1}=0$ on the boundary to write it as
\begin{equation}\label{q:idd2}
   -2\,(\Delta\omega^{n+1},\delta\omega^{n+1})
	= |\gb\omega^{n+1}|^2 - |\gb\omega^n|^2 + |\gb\delta\omega^{n+1}|^2.
\end{equation}
We bound the nonlinear term as
\begin{equation}\begin{aligned}
   4\,\bigl|(\dy(2\psi^n-\psi^{n-1},2&\omega^n-\omega^{n-1}),\delta\omega^{n+1})\bigr|\\
	&\le c\,|2\gb\psi^n-\gb\psi^{n-1}|_{L^\infty}^{}|2\gb\omega^n-\gb\omega^{n-1}|_{L^2}^{}|\delta\omega^{n+1}|_{L^2}^{}\\
	&\le \sfrac18\,|\delta\omega^{n+1}|^2 + c\,|2\gb\omega^n-\gb\omega^{n-1}|^4.
\end{aligned}\end{equation}
Bounding the buoyancy terms by Cauchy--Schwarz, we arrive at
\begin{equation}\begin{aligned}
   2\,|\delta\omega^{n+1}|^2 &+ \sfrac13\,|3\delta\omega^{n+1}-\delta\omega^n|^2
	+ 2\ppr k\,|\gb\omega^{n+1}|^2 + 2\ppr k\,|\gb\delta\omega^{n+1}|^2\\
	&\le \sfrac13\,|\delta\omega^n|^2 + 2\ppr k\,|\gb\omega^n|^2
	+ ck^2\,|2\gb\omega^n-\gb\omega^{n-1}|^4\\
	&\qquad+ c\ppr^2k^2\,\bigl(|\dy_x T^{n+1}|^2 + |\dy_x S^{n+1}|^2\bigr)\\
	&\le \sfrac13\,|\delta\omega^n|^2 + c(\parm)\bigl(kM_1 + k^2M_1^2\bigr).
\end{aligned}\end{equation}
It is now clear that, since $\delta\omega^1$ is bounded in $L^2$,
we have for large $nk$
\begin{equation}
   |\delta\omega^n|^2 \le k\,c(\parm)(M_1 + kM_1^2).
\end{equation}

Similarly for $\Sh^n$, we multiply (\ref{q:eih}c) by $4k\delta\Sh^{n+1}$
to find
\begin{equation}\begin{aligned}
   3\,|\delta\Sh^{n+1}|^2 &+ \sfrac13\,|3\delta\Sh^{n+1}-\delta\Sh^n|^2
	= \sfrac13\,|\delta\Sh^n|^2 + 4k\pts\,(\Delta\Sh^{n+1}\!+\Delta\sq,\delta\Sh^{n+1})\\
&- 4k\,(\dy(2\psi^n\!-\psi^{n-1},2\Sh^n\!-\Sh^{n-1}\!+\sq),\delta\Sh^{n+1}).
\end{aligned}\end{equation}
Bounding the nonlinear term as we did for $\omega^n$,
\begin{equation}\begin{aligned}
   &4\,\bigl|(\dy(2\psi^n\!-\psi^{n-1},2\Sh^n\!-\Sh^{n-1}\!+\sq),\delta\Sh^{n+1})\bigr|\\
	&\qquad\le \sfrac18\,|\delta\Sh^{n+1}|^2 + c\,|2\gb\omega^n\!-\gb\omega^{n-1}|^2(|2\gb\Sh^n\!-\gb\Sh^{n-1}|^2 + |\gb\sq|^2),
\end{aligned}\end{equation}
and the linear terms as we did with $\omega^n$, we arrive at
\begin{align}
   \!\!\!2\,&|\delta\Sh^{n+1}|^2 + \sfrac13\,|3\delta\Sh^{n+1}-\delta\Sh^n|^2
	+ 2\pts k\,|\gb\Sh^{n+1}|^2 + 2\pts k\,|\gb\delta\Sh^{n+1}|^2 \notag\\
	&\quad\le \sfrac13\,|\delta\Sh^n|^2 + 2\pts k\,|\gb\Sh^n|^2
	+ ck^2 |2\gb U^n\!-\gb U^{n-1}|^4 + c(\pts)k^2 (|\gb\sq|^4 + |\Delta\sq|^2),
\end{align}
whence
\begin{equation}
   |\delta\Sh^n|^2 \le k\,c(\parm)(M_1 + kM_1^2)
	\qquad\textrm{for large }nk.
\end{equation}
Obviously a similar bound holds for $\delta\Th^n$, so we conclude that
\begin{equation}\label{q:delU}
   |\delta U^n|^2 \le k\,c(\parm)(M_1 + kM_1^2) =: k\mt_\delta
	\qquad\textrm{for large }nk.
\end{equation}

\medskip
By taking difference of (\ref{q:ei}a), we find
\begin{equation}\label{q:del0}\begin{aligned}
   &\frac{3\delta\omega^{n+1}-4\delta\omega^n+\delta\omega^{n-1}}{2k}
	+ \dy(2\psi^{n-1}-\psi^{n-2},2\delta\omega^n-\delta\omega^{n-1})\\
	&\quad+ \dy(2\delta\psi^n-\delta\psi^{n-1},2\omega^n-\omega^{n-1})
	= \ppr\bigl\{\Delta\delta\omega^{n+1} + \dy_x\delta T^{n+1} - \dy_x\delta S^{n+1}\bigr\}.
\end{aligned}\end{equation}
Multiplying this by $2k\delta\omega^{n+1}$ and using \eqref{q:hs00},
we have
\begin{equation}\begin{aligned}
   &\gn{\delta\omega^n,\delta\omega^{n+1}}_{\nu k}^2 - \nu k\,|\delta\omega^{n+1}|^2 + \frac{|(1+\nu k)\delta\omega^{n+1}\!-2\delta\omega^n+\delta\omega^{n-1}|^2}{2(1+\nu k)} + kI\\
	&\quad= \frac{\gn{\delta\omega^{n-1},\delta\omega^n}_{\nu k}^2}{1+\nu k}
	- 2\ppr k\,|\gb\delta\omega^{n+1}|^2
	+ 2\ppr k\,(\dy_x\delta T^{n+1}\!-\dy_x\delta S^{n+1},\delta\omega^{n+1}).
\end{aligned}\end{equation}
Here $I=I_1+I_2$ denotes the nonlinear terms, which we bound as
\begin{equation}\begin{aligned}
   |I_1| &\le c\,|2\gb\psi^{n-1}\!-\gb\psi^{n-2}|_{L^\infty}^{}|\gb\delta\omega^{n+1}|_{L^2}^{}|2\delta\omega^n\!-\delta\omega^{n-1}|_{L^2}^{}\\
	&\le \frac{\ppr}8\,|\gb\delta\omega^{n+1}|^2 + \frac{c}\ppr\,|2\gb\omega^{n-1}-\gb\omega^{n-2}|^2|2\delta\omega^n-\delta\omega^{n-1}|^2\\
   |I_2| &\le c\,|2\gb\delta\psi^n-\gb\delta\psi^{n-1}|_{L^4}^{}|2\omega^n-\omega^{n-1}|_{L^4}^{}|\gb\delta\omega^{n+1}|_{L^2}^{}\\
	&\le \frac{\ppr}8\,|\gb\delta\omega^{n+1}|^2 + \frac{c}\ppr\,|2\delta\omega^n-\delta\omega^{n-1}|^2|2\gb\omega^n-\gb\omega^{n-1}|^2.
\end{aligned}\end{equation}
Bounding the linear terms as
\begin{equation}
   \bigl|(\dy_x\delta T^{n+1}\!-\dy_x\delta S^{n+1},\delta\omega^{n+1})\bigr|
	\le \sfrac14\,|\gb\delta\omega^{n+1}|^2 + 2\,|\delta T^{n+1}|^2 + 2\,|\delta S^{n+1}|^2
\end{equation}
and using \eqref{q:delU}, we obtain
\begin{equation}\label{q:del1}\begin{aligned}
   \gn{\delta\omega^n,\delta\omega^{n+1}}_{\nu k}^2 &+ \ppr k\,|\gb\delta\omega^{n+1}|^2\\
	&\le \frac1{1+\nu k}\gn{\delta\omega^{n-1},\delta\omega^n}_{\nu k}^2
	+ k^2c(\parm)\tilde M_\delta(1+M_1).
\end{aligned}\end{equation}
Integrating this and the analogous expressions for $\delta T^n$
and $\delta S^n$, we obtain \eqref{q:bddU} for $nk$ large.

\medskip
To prove \eqref{q:h2}, we note that (\ref{q:ei}b) implies
\begin{equation}\begin{aligned}
   |\Delta T^{n+1}| &\le |\dy(2\psi^n-\psi^{n-1},2T^n-T^{n-1})|
	+ \frac{|3\delta T^{n+1} - \delta T^n|}{2k}\\
	&\le c\,|2\gb\omega^n-\gb\omega^{n-1}|\,|2\gb T^n-\gb T^{n-1}|
	+ \frac{3|\delta T^{n+1}| + |\delta T^n|}{2k}.
\end{aligned}\end{equation}
Since the right-hand side has been bounded (independently of $k$ for the
first term and by $Mk$ for the second) on the
attractor $\Attr_k$, it follows that $|\Delta T^n|$ is uniformly bounded
on $\Attr_k$ as well.
Clearly similar $H^2$ bounds also hold for $S^n$ and $\omega^n$,
proving \eqref{q:h2} and the Theorem.
\end{proof}


\appendix
\section{2d Navier--Stokes equations}

In this appendix we present an alternate derivation of the boundedness
results in \cite{wxm:12}, without using the Wente-type estimate of
\cite{kim:09} but requiring slightly more regular initial data.
In principle these could be obtained following the proofs of
Theorems \ref{t:h1}~and~\ref{t:h2} above,
but the computation is much cleaner in this case (mostly due to the
periodic boundary conditions) so we present it separately.

The system is the 2d Navier--Stokes equations
\begin{equation}\label{q:adwdt}
   \frac{3\omega^{n+1}-4\omega^n+\omega^{n-1}}{2k}
	+ \dy(2\psi^n-\psi^{n-1},2\omega^n-\omega^{n-1})
	= \mu\Delta\omega^{n+1} + f^n
\end{equation}
with periodic boundary conditions.
It is clear that $\omega^n$ has zero integral over $\Dom$,
and we define $\psi^n$ uniquely by the zero-integral condition.
These imply \eqref{q:cpoi}--\eqref{q:normeq},
which we will use below without further mention.
Assuming that the initial data $\omega^0$, $\omega^1\in H^{1/2}$ (in fact,
we only need $H^\epsilon$ for any $\epsilon>0$, but will write $H^{1/2}$ for
concreteness), we derive uniform bounds for $\omega^n$ in $L^2$, $H^1$ and $H^2$.

Assuming for now the uniform bound
\begin{equation}\label{q:ahalf}
   |\omega^n|_{H^{1/2}}^2 \le k^{-1/2} M_\omega(\cdots)
   \qquad\textrm{for }n\in\{2,3,\cdots\},
\end{equation}
we multiply \eqref{q:adwdt} by $2k\omega^{n+1}$ in $L^2$, use \eqref{q:hs00}
and estimate as before,
\begin{equation}\begin{aligned}
   &\hskip-20pt\gn{\omega^n,\omega^{n+1}}_{\nu k}^2 - \nu k\,|\omega^{n+1}|^2
	+ 2\mu k\,|\gb\omega^{n+1}|^2 + \frac{|(1+\nu k)\omega^{n+1}\!-2\omega^n\!+\omega^{n-1}|^2}{2(1+\nu k)}\\
   &= \frac{\gn{\omega^{n-1},\omega^n}_{\nu k}^2}{1+\nu k} + 2k\,(f^n,\omega^{n+1})\\
	&\qquad- 2k\,(\dy(2\psi^n-\psi^{n-1},\omega^{n+1}),(1+\nu k)\omega^{n+1}-2\omega^n+\omega^{n-1})\\
   &\le \frac{\gn{\omega^{n-1},\omega^n}_{\nu k}^2}{1+\nu k}
	+ \frac{\mu k}2\,|\gb\omega^{n+1}|^2 + \frac{ck}\mu\,|f^n|_{H^{-1}}^2\\
	&\qquad+ \frac{ck}\mu\,|2\gb\psi^n-\gb\psi^{n-1}|_{L^\infty}^2|(1+\nu k)\omega^{n+1}-2\omega^n+\omega^{n-1}|^2,
\end{aligned}\end{equation}
giving (as before, we require $k\le 1/\nu$)
\begin{equation}\begin{aligned}
   \gn{\omega^n,\omega^{n+1}}_{\nu k}^2 &- \nu k\,|\omega^{n+1}|^2
	+ \frac{3\mu k}2\,|\gb\omega^{n+1}|^2
   \le \frac{\gn{\omega^{n-1},\omega^n}_{\nu k}^2}{1+\nu k}
	+ \frac{ck}\mu\,|f^n|_{H^{-1}}^2\\
	&+ |(1+\nu k)\omega^{n+1}\!-2\omega^n\!+\omega^{n-1}|^2
		\bigl({c_3k^{1/2}M_\omega}/\mu - \sfrac14\bigl).
\end{aligned}\end{equation}
Setting $\nu=\mu/(2\cpoi)$ and imposing the timestep restriction
\begin{equation}\label{q:adt}
   k \le k_0 := \min\{\mu^2/(4c_3 M_\omega)^2,1/\nu\},
\end{equation}
this gives
\begin{equation}\label{q:aidl2}
   \gn{\omega^n,\omega^{n+1}}_{\nu k}^2  + \mu k\,|\gb\omega^{n+1}|^2
	\le \frac{\gn{\omega^{n-1},\omega^n}_{\nu k}^2}{1+\nu k}
	+ \frac{ck}\mu\,|f^n|_{H^{-1}}^2.
\end{equation}
Integrating using the Gronwall lemma, we arrive at the $L^2$ bound
\begin{equation}\label{q:abdl2}\begin{aligned}
   \gn{\omega^{n+1},\omega^{n+2}}_{\nu k}^2 &+ \mu k\,|\gb\omega^{n+2}|^2
	\le \ex^{-\nu nk/2}\gn{\omega^0,\omega^1}_{\nu k}^2
		+ \frac{c}{\mu^2}\,\tssup_j|f^j|_{H^{-1}}^2\\
	&\le \gn{\omega^0,\omega^1}_{\nu k}^2
		+ \frac{c}{\mu^2}\,\tssup_j|f^j|_{H^{-1}}^2 =: M_0.
\end{aligned}\end{equation}
The hypothesis \eqref{q:ahalf} is now recovered by interpolation as before,
\begin{equation}\begin{aligned}
   |\omega^n|_{H^{1/2}}^2 &\le c\,|\omega^n|\,|\gb\omega^n|
	\le c\,\gn{\omega^{n-1},\omega^n}_{\nu k}^{}|\gb\omega^n|\\
	&\le c\,(\mu k)^{-1/2} \bigl(\gn{\omega^0,\omega^1}_{\nu k}^2 + (1/\mu+1/\mu^2)\,\tssup_j |f^j|_{H^{-1}}^2\bigr).
\end{aligned}\end{equation}
Summing \eqref{q:aidl2}, we find
\begin{equation}\label{q:abdl2h1}
   \mu k\,\tssum_{j=n+1}^{n+\nk{1/k}} |\gb\omega^j|^2 \le \gn{\omega^{n-1},\omega^n}_{\nu k}^2
	+ c_\mu\,\tssup_j |f^j|_{H^{-1}}^2.
\end{equation}
It is clear that both bounds \eqref{q:abdl2} and \eqref{q:abdl2h1}
can be made independent of the initial data for sufficiently large
time, $nk\ge t_0(\omega^0,\omega^1;f,\mu)$.

For the $H^1$ estimate, we multiply \eqref{q:adwdt} by $-2k\Delta\omega^{n+1}$
in $L^2$ and use \eqref{q:hs00}.
Writing the nonlinear term as
\begin{equation}\begin{aligned}
   \!N_1 := (\dy(2\psi^n&-\psi^{n-1},2\omega^n-\omega^{n-1}),\Delta\omega^{n+1})\\
	&= (\dy(2\gb\psi^n-\gb\psi^{n-1},\gb\omega^{n+1}),2\omega^n-\omega^{n-1})\\
	&- (\dy(2\psi^n-\psi^{n-1},\gb\omega^{n+1}),\gb((1+\nu k)\omega^{n+1}-2\omega^n+\omega^{n-1}))
\end{aligned}\end{equation}
and bounding the terms as
\begin{equation}\begin{aligned}
   \hskip-8pt|N_1| &\le c\,|2\omega^n-\omega^{n-1}|_{L^4}^{}|\gb^2\omega^{n+1}|_{L^2}^{}|2\omega^n-\omega^{n-1}|_{L^4}^{}\\
	&\quad+ c\,|2\gb\psi^n\!-\gb\psi^{n-1}|_{L^\infty}^{}|\gb^2\omega^{n+1}|_{L^2}^{}|\gb((1+\nu k)\omega^{n+1}\!-2\omega^n\!+\omega^{n-1})|_{L^2}^{}\\
	&\le \frac\mu2\,|\Delta\omega^{n+1}|^2 + \frac{c}\mu\,|2\omega^n-\omega^{n-1}|^2|2\gb\omega^n-\gb\omega^{n-1}|^2\\
	&\quad+ \frac{ck^{-1/2}}\mu\,M_\omega\,|\gb((1+\nu k)\omega^{n+1}\!-2\omega^n\!+\omega^{n-1})|^2,
\end{aligned}\end{equation}
we find the differential inequality, using the bound \eqref{q:abdl2},
\begin{equation}\label{q:aidh1}\begin{aligned}
   &\gn{\gb\omega^n,\gb\omega^{n+1}}_{\nu k}^2 + \mu k\,|\Delta\omega^{n+1}|^2
	\le \gn{\gb\omega^{n-1},\gb\omega^n}_{\nu k}^2 \bigl(1 + ck\,M_0/\mu\bigr)\\
	&\qquad+ |\gb((1+\nu k)\omega^{n+1}\!-2\omega^n\!+\omega^{n-1})|^2 \bigl({c_3k^{1/2}M_\omega}/\mu - \sfrac14\bigr)
	+ ck\,|f^n|^2/\mu.
\end{aligned}\end{equation}
Using the earlier timestep restriction \eqref{q:adt}, we can suppress the
second term on the r.h.s.
Thanks to \eqref{q:abdl2h1}, for any $n\in\{0,1,\cdots\}$ we can find
$n_*\in\{n,\cdots,n+\nk{1/k}\}$ such that
$\gn{\gb\omega^{n_*},\gb\omega^{n_*+1}}_{\nu k}^2\le c(\mu)\,\bigl(\gn{\omega^0,\omega^1}_{\nu k}^2 + \sup_j|f^j|_{H^{-1}}^2\bigr)$.
Arguing as before, we can use this to integrate \eqref{q:aidh1} to give us
a uniform $H^1$ bound
\begin{equation}\label{q:abdh1}
   \gn{\gb\omega^n,\gb\omega^{n+1}}_{\nu k}^2 \le M_1(|\gb\omega^0|,|\gb\omega^1|;\mu,\tssup_j|f^j|)
\end{equation}
valid for all $n\in\{0,1,\cdots\}$.
Moreover, $M_1$ can be made independent of the initial data
$|\gb\omega^0|$, $|\gb\omega^1|$ for sufficiently large $n$;
in fact, we do not even need $\omega^0$, $\omega^1\in H^1$,
although we still need them to be in $H^\epsilon$ for the timestep
restriction~\eqref{q:adt}.
Summing \eqref{q:aidh1} and using \eqref{q:abdh1}, we find
\begin{equation}
   \mu k\,\tssum_{j=n+1}^{n+\nk{1/k}} |\Delta\omega^j|^2
	\le \mt_1(\tssup_j|f^j|;\mu)
	\qquad\textrm{for all }nk\ge t_1(\omega^0,\omega^1,f;\mu).
\end{equation}

Similarly, for the $H^2$ estimate, we multiply \eqref{q:adwdt} by
$2k\Delta^2\omega^{n+1}$ in $L^2$ and write the nonlinear term as
\begin{equation}\begin{aligned}
   N_2 := (\dy(2\psi^n&-\psi^{n-1},2\omega^n-\omega^{n-1}),\Delta^2\omega^{n+1})\\
	&= -(\dy(2\gb\psi^n-\gb\psi^{n-1},2\omega^n-\omega^{n-1}),\gb\Delta\omega^{n+1})\\
	&\quad- (\dy(2\psi^n-\psi^{n-1},2\gb\omega^n-\gb\omega^{n-1}),\gb\Delta\omega^{n+1}).
\end{aligned}\end{equation}
Bounding this as
\begin{equation}\begin{aligned}
   |N_2| &\le c\,|2\omega^n-\omega^{n-1}|_{L^\infty}^{}|2\gb\omega^n-\gb\omega^{n-1}|_{L^2}^{}|\gb\Delta\omega^{n+1}|_{L^2}^{}\\
	&\quad+ c\,|2\gb\psi^n-\gb\psi^{n-1}|_{L^\infty}^{}|2\gb^2\omega^n-\gb^2\omega^{n-1}|_{L^2}^{}|\gb\Delta\omega^{n+1}|_{L^2}^{}\\
	&\le \frac\mu2\,|\gb\Delta\omega^{n+1}|^2 + \frac{c}\mu\,|2\gb\omega^n-\gb\omega^{n-1}|^2\gn{\Delta\omega^{n-1},\Delta\omega^n}_{\nu k}^2,
\end{aligned}\end{equation}
we arrive at the differential inequality
\begin{equation}\begin{aligned}
   \gn{\Delta\omega^n,\Delta\omega^{n+1}}_{\nu k}^2
	&+ \mu k\,|\gb\Delta\omega^{n+1}|^2\\
	&\le \gn{\Delta\omega^{n-1},\Delta\omega^n}_{\nu k}^2\bigl(1 + ck M_1/\mu\bigr)
	+ ck |\gb f^n|^2/\mu.
\end{aligned}\end{equation}
As with \eqref{q:aidh1}, this can be integrated to obtain the uniform bound
\begin{equation}\label{q:aidh2}
   \gn{\Delta\omega^n,\Delta\omega^{n+1}}_{\nu k}^2
	\le M_2(\tssup_j|\gb f^j|;\mu)
\end{equation}
valid whenever $nk\ge t_2(\omega^0,\omega^1,f;\mu)$.

To bound the difference $\delta\omega^n:=\omega^n-\omega^{n-1}$,
we write \eqref{q:adwdt} as
\begin{equation}
   \frac{3\delta\omega^{n+1}-\delta\omega^n}{2k}
	+ \dy(2\psi^n-\psi^{n-1},2\omega^n-\omega^{n-1})
	= \mu\Delta\omega^{n+1} + f^n.
\end{equation}
Multiplying by $4k\delta\omega^{n+1}$ and using \eqref{q:idd1}
and \eqref{q:idd2}, we find
\begin{equation}\begin{aligned}
   3|\delta\omega^{n+1}|^2 &+ \sfrac13|\delta\omega^{n+1}-\delta\omega^n|^2
	= \sfrac13|\delta\omega^n|^2\\
	&+ 2\mu k|\gb\omega^n|^2
	- 2\mu k|\gb\omega^{n+1}|^2 - 2\mu k|\gb\delta\omega^{n+1}|^2\\
	&- 4k(\dy(2\psi^n-\psi^{n-1},2\omega^n-\omega^{n-1}),\delta\omega^{n+1})
	+ 4k(f^n,\delta\omega^{n+1}).
\end{aligned}\end{equation}
Bounding the nonlinear term and suppressing harmless terms, we arrive at
\begin{equation}\begin{aligned}
   2|\delta\omega^{n+1}|^2 &\le \sfrac13|\delta\omega^n|^2
	+ 2\mu k|\gb\omega^n|^2\\
	&\quad+ ck^2|2\gb\psi^n-\gb\psi^{n-1}|_{L^\infty}^2|2\gb\omega^n-\gb\omega^{n-1}|^2
	+ \frac{ck^2}\mu|f^n|_{H^{-1}}^2.
\end{aligned}\end{equation}
Since the r.h.s.\ has been bounded uniformly for large $nk$, we conclude that
\begin{equation}
   |\delta\omega^n|^2 \le k \hat M_0(f,\mu)
\end{equation}
for $nk$ sufficiently large.
Arguing as in \eqref{q:del0}--\eqref{q:del1}, we can improve the bound
on $|\delta\omega^n|$ to ${\sf O}(k)$.



\nocite{foias-manley-temam:87b}
\nocite{hsia-ma-wang:08}

\nocite{balmforth-al:06}
\nocite{radko-stern:00}


\end{document}